\documentclass[11pt,reqno]{amsart}
\usepackage[T1]{fontenc}
\usepackage{newunicodechar}
\usepackage{lmodern}
\usepackage{amsmath,amssymb,amsthm,amsfonts}
\usepackage{mathrsfs}
\usepackage{bm}
\usepackage{enumitem}
\usepackage{hyperref}
\usepackage[nameinlink,capitalize]{cleveref}
\usepackage{geometry}
\usepackage{cite}
\usepackage{color}
\usepackage{soul}

\numberwithin{equation}{section}

\newtheorem{theorem}{Theorem}[section]

\newtheorem{proposition}[theorem]{Proposition}
\newtheorem{corollary}[theorem]{Corollary}

\theoremstyle{definition}
\newtheorem{definition}[theorem]{Definition}

\theoremstyle{remark}

\newcommand{\R}{\mathbb{R}}

\newcommand{\Ric}{\operatorname{Ric}}
\newcommand{\Scal}{\operatorname{Scal}}

\title[An Information--Theoretic Reconstruction of Curvature]{An Information-Theoretic Reconstruction of Curvature}

\author{Amandip Sangha}
\thanks{preprint}
\address{The Climate and Environmental Research Institute NILU}
\email{asan@nilu.no}

\date{}

\begin{document}

\begin{abstract}
We develop an intrinsic information--theoretic approach for recovering 
Riemannian curvature from the small--time behaviour of heat diffusion.  
Given a point and a two--plane in the tangent space, we compare the heat mass 
transported along that plane with its Euclidean counterpart using the relative 
entropy of finite measures.  We show that the leading small--time distortion of 
this directional entropy encodes precisely the local curvature of the manifold.  
In particular, the planar information imbalance determines both the scalar 
curvature and the sectional curvature at a point, and assembling these 
directional values produces a bilinear tensor that coincides exactly with the 
classical Riemannian curvature operator.

The method is entirely analytic and avoids Jacobi fields, curvature identities, 
or variational formulas.  Curvature appears solely through the behaviour of 
heat flow under the exponential map, providing a new viewpoint in which 
curvature is realized as an infinitesimal information defect of diffusion.  
This perspective suggests further connections between geometric analysis and 
information theory and offers a principled analytic mechanism for detecting and 
reconstructing curvature using only heat diffusion data.
\end{abstract}

\maketitle
\tableofcontents

\section{Introduction}

The purpose of this paper is to develop an intrinsic, information–-theoretic
framework for reconstructing Riemannian curvature from the small–time behaviour
of heat diffusion.  Classically, sectional curvature is defined through
second–order variations of area, Jacobi fields, or comparison geometry, while
analytic approaches extract curvature from the short–time expansion of the heat
kernel in geodesic normal coordinates.  Our aim is to show that the full
curvature operator may be recovered from the \emph{directional
information imbalance} of heat flow, with no reference to Jacobi fields or
connection coefficients.

The construction begins by restricting the heat flow emanating from a point
\(x\in M\) to a two-–plane \(\sigma\subset T_xM\) via the exponential map.  This
produces an \emph{unnormalized} heat mass measure \(\mu^{M,\sigma}_t\) on
\(\sigma\).  For comparison, the Euclidean heat flow on the tangent space
induces an unnormalized heat mass \(\mu^{\mathbb{R}^n,\sigma}_t\).  Their total
masses differ at order \(t^{1-n/2}\), so it is natural to renormalize both by
the same scalar
\[
M_t := \mu^{\mathbb{R}^n,\sigma}_t(\sigma),
\quad
\nu_t := M_t^{-1}\,\mu^{\mathbb{R}^n,\sigma}_t,
\quad
\tilde\mu^{M,\sigma}_t := M_t^{-1}\,\mu^{M,\sigma}_t.
\]
The reference measure \(\nu_t\) is then exactly the centred Gaussian \emph{probability}
measure on the plane \(\sigma\), and the Radon-–Nikodym derivative
\(\tfrac{d\tilde\mu^{M,\sigma}_t}{d\nu_t}\) coincides with the density ratio of the
two unnormalized heat masses.

We compare these normalized heat distributions using the relative entropy
\[
D_\sigma(t)
 := H(\tilde\mu^{M,\sigma}_t \,\|\, \nu_t)
 = \int_\sigma \log\!\left(\frac{d\tilde\mu^{M,\sigma}_t}{d\nu_t}\right)
                d\tilde\mu^{M,\sigma}_t,
\]
which measures the information defect incurred when attempting to describe heat
diffusion on \((M,g)\) by its Euclidean counterpart along the plane \(\sigma\).
Because heat kernels concentrate on the scale \(|v| = O(\sqrt{t})\), only the
lowest–-order geometric corrections of the metric and exponential map influence
the behaviour of \(D_\sigma(t)\) as \(t\downarrow 0\).

Our main analytic result establishes that \(D_\sigma(t)\) admits a signed,
linear small–time expansion whose coefficient encodes the curvature of
\(\sigma\).  More precisely,
\[
D_\sigma(t)
 = \tfrac16\Scal(x)\,t + C\,K(x,\sigma)\,t + O(t^{3/2}),\qquad
C = -\tfrac23.
\]
The isotropic term is determined by the scalar curvature at \(x\), while the
anisotropic component encodes the sectional curvature of the plane \(\sigma\).
Thus the directional limit \(\lim_{t\downarrow 0} D_\sigma(t)/t\) determines
\(K(x,\sigma)\), and averaging these limits over all coordinate planes recovers
\(\Scal(x)\).

Motivated by this expansion, we assemble the directional coefficients
\(D_\sigma(t)/t\) into a bilinear form \(I_x\) on \(\Lambda^2T_xM\).  We prove that
\(I_x\) coincides exactly with the classical Riemann curvature operator
\(R_x : \Lambda^2T_xM \to \Lambda^2T_xM\), in the sense that for any simple
bivector \(\omega = u\wedge v\),
\[
I_x(\omega,\omega) = K(x,\sigma_\omega)
 = \langle R_x\omega,\omega\rangle.
\]
Thus the family \(\{D_\sigma(t)\}_{\sigma\subset T_xM}\) encodes the complete
Riemann curvature tensor at \(x\).  All geometric information obtained in this
way arises purely from the interaction between the heat kernel and the geometry
of the exponential map; no curvature identities or variational formulas enter
the argument.

This yields an intrinsic information–-theoretic interpretation of curvature: it
is the infinitesimal rate at which heat diffusion, restricted to a
two–dimensional slice, loses information relative to Euclidean space.  The
framework is entirely local and analytic, and suggests new connections between
geometric analysis and information theory, as well as potential applications to
geometric inequalities, stability questions, and curvature reconstruction from
diffusion data.

The information-–theoretic viewpoint developed here is inspired by the author's
previous work~\cite{SanghaHeatFlowIsoperimetry}, in which entropy dissipation under
heat flow was used to give an information–-theoretic proof of classical
isoperimetric inequalities.  That paper revealed how diffusion can be interpreted
as a noisy communication channel whose initial information-–loss rate encodes
geometric boundary data.  The present work extends this philosophy from boundary
geometry to intrinsic curvature: instead of perimeter appearing in the leading
term of an entropy expansion, we show that the full Riemann curvature operator
arises from the small–time information defect of heat flow restricted to
two–dimensional slices.  This provides a higher-–order, curvature-–sensitive
analogue of the analytic mechanism introduced in~\cite{SanghaHeatFlowIsoperimetry}.

We now give a brief outline of the remainder of the paper. Section~\ref{sec:preliminaries} recalls the necessary geometric, analytic, and information--theoretic
background.  
Section~\ref{sec:info-functional} constructs the curvature--information functional and proves its
small--time expansion.  
Section~\ref{sec:info-tensor} introduces the curvature--information tensor and shows that it
recovers the classical curvature operator.  
Section~\ref{sec:structural-props} develops structural properties, including information--theoretic
expressions for Ricci and scalar curvature. Section~\ref{sec:conclusion} ends the paper with concluding remarks and discusses future directions.

% ============================================================
\section{Preliminaries}
\label{sec:preliminaries}
Background for Section~\ref{subsec:riem-geometry} may be found in
\cite{Jost2017,Rosenberg1997}, for Section~\ref{subsec:heat-kernel} in
\cite{BGV1992,Grigoryan2009}, and for Section~\ref{subsec:relative-entropy} in
\cite{CoverThomas2006,DemboZeitouni1998}.

\subsection{Riemannian geometry and sectional curvature}
\label{subsec:riem-geometry}

Let $(M,g)$ be an $n$-dimensional smooth Riemannian manifold.  
For $x\in M$, we denote by $\exp_x : T_xM \to M$ the exponential map and use the identification
\[
T_xM \cong \R^n
\]
via an orthonormal basis $(e_1,\dots,e_n)$.

The Riemann curvature tensor is defined by
\[
R(u,v)w
= \nabla_u\nabla_v w - \nabla_v\nabla_u w - \nabla_{[u,v]}w,
\qquad u,v,w\in T_xM,
\]
with components
\[
R_{ijkl} = g\bigl(R(e_i,e_j)e_k, e_l\bigr).
\]
The Ricci tensor is
\[
\Ric_{ij} = R_{kikj}, 
\qquad
\Scal = \Ric_{ii}.
\]

For a $2$--plane $\sigma \subset T_xM$ spanned by orthonormal vectors $u,v$, the sectional curvature is
\begin{equation}
\label{eq:sectional-curvature}
K(x,\sigma)
= \langle R(u,v)v, u\rangle
= R(u,v,v,u)
= R_{1212},
\end{equation}
after choosing a frame with $e_1=u$, $e_2=v$.

Throughout the paper we use \emph{geodesic normal coordinates} at $x$, so that
\[
g_{ij}(0)=\delta_{ij},
\qquad
\partial_k g_{ij}(0)=0.
\]
The second-order Taylor expansion of the metric is the classical expression
\begin{equation}
\label{eq:metric-expansion}
g_{ij}(y)
= \delta_{ij} - \frac{1}{3} R_{ikjl}(x)\, y^k y^l + O(|y|^3),
\qquad y\to 0\in T_xM.
\end{equation}

The inverse metric satisfies
\begin{equation}
\label{eq:inverse-metric-expansion}
g^{ij}(y)
= \delta^{ij} + \frac{1}{3} R_{ikjl}(x)\, y^k y^l + O(|y|^3).
\end{equation}

Let $|g(y)|=\det(g_{ij}(y))$. In normal coordinates,
\begin{equation}
\label{eq:det-expansion}
|g(y)|^{1/2}
= 1 - \frac{1}{6} \Ric_{kl}(x)\, y^k y^l + O(|y|^3).
\end{equation}

If $v\in\sigma=\mathrm{span}\{e_1,e_2\}$, the Jacobian of the restriction of $\exp_x$ to $\sigma$ has the expansion \cite{Gray1973}
\begin{equation}
\label{eq:jacobian-plane}
J_x(v)
= 1 - \frac{1}{6}K(x,\sigma)\,|v|^2 + O(|v|^3),
\qquad |v|\to0,
\end{equation}
where $K(x,\sigma)=R_{1212}$ in the frame $(e_1,e_2)$.

Restricting the metric to $\sigma$ yields, in these coordinates,
\begin{equation}
\label{eq:restricted-metric}
g_{ab}(v)
= \delta_{ab} - \frac{1}{3} K(x,\sigma)\, v^c v^c\, \delta_{ab}
+ O(|v|^3),
\qquad a,b\in\{1,2\}.
\end{equation}

The Gaussian curvature of this induced $2$-metric at $0$ coincides with $K(x,\sigma)$.

\medskip

These expansions will be used in section \ref{sec:info-functional} to compare the heat kernel on $(M,g)$ with the Euclidean heat kernel on $(\sigma,\delta)$, and to extract $K(x,\sigma)$ from the second-order information defect.

\subsection{Heat kernel and small-time asymptotics}\label{subsec:heat-kernel}

Let $(M,g)$ be a smooth, complete Riemannian manifold. The heat kernel $p_t^M(x,y)$ is the fundamental solution of the heat equation.
More precisely, it is the unique smooth function
\[
p_t^M : (0,\infty)\times M\times M \longrightarrow (0,\infty)
\]
satisfying
\[
\partial_t p_t^M(x,y) = \Delta_y p_t^M(x,y), \qquad
\lim_{t\downarrow 0} p_t^M(x,y)\, d\mathrm{vol}_y = \delta_x ,
\]
where $\Delta$ is the Laplace--Beltrami operator.  
For each fixed $x\in M$ and $t>0$, the quantity $p_t^M(x,y)\,d\mathrm{vol}_y$
is the heat distribution at time $t$ emanating from $x$. For fixed $x\in M$,
we work in geodesic normal coordinates so that $y=\exp_x(v)$, $v\in T_xM$.  The
classical parametrix construction \cite{BGV1992}, \cite{Grigoryan2009} yields the short-time expansion
\[
p_t^M(x,\exp_x v)
 = (4\pi t)^{-n/2}
   \exp\!\Bigl(-\frac{|v|^2}{4t}\Bigr)
   \Bigl(u_0(x,v) + t\,u_1(x,v) + t^2 u_2(x,v) + O(t^3)\Bigr),
\]
valid uniformly for $|v|\le r$ in a fixed normal neighbourhood of~$x$.  The
coefficients $u_j$ are smooth and determined recursively by transport equations
along radial geodesics; in particular
\[
u_0(x,0)=1,
\qquad
u_1(x,0)=\tfrac{1}{6}\,\Scal(x).
\]

The dominant Gaussian factor captures Euclidean diffusion at scale $\sqrt{t}$,
while the coefficient functions record curvature corrections.  When restricting to
a $2$--plane $\sigma\subset T_xM$, we will make use of this expansion with $v$
ranging over $\sigma$ and $|v|=O(\sqrt{t})$, in which regime only the lowest--order
Taylor coefficients of $u_j$ contribute.  Together with the Jacobian expansion of
the exponential map along~$\sigma$, these asymptotics determine the first-order
information defect of heat diffusion relative to its flat counterpart.

\subsection{Relative entropy for heat mass distributions}
\label{subsec:relative-entropy}
We briefly recall the notion of relative entropy in the setting required for our
construction.  Let $\mu$ and $\nu$ be finite Borel measures on a measurable space
with $\mu\ll\nu$.  The \emph{relative entropy} of $\mu$ with respect to $\nu$ is
\[
H(\mu\Vert\nu)
 := \int \log\!\Bigl(\frac{d\mu}{d\nu}\Bigr)\,d\mu,
\]
whenever the expression is well defined.  This functional measures the information
defect of $\mu$ relative to the reference measure $\nu$ and reduces to the usual
Kullback--Leibler divergence when both are probability measures.  In contrast to
the probabilistic setting, no normalization is imposed here; the total masses
$\mu(X)$ and $\nu(X)$ enter the quantity through the logarithmic density ratio.

For a one-parameter family $\mu_t=(1+t\,a_t+O(t^{3/2}))\,\nu_t$ with $a_t$ uniformly
bounded on compact sets, a Taylor expansion of $\log(1+t\,a_t)$ yields the
first-order asymptotic law
\[
H(\mu_t\Vert\nu_t)
 = t\!\int a_t\,d\nu_t + O(t^{3/2}).
\]
Thus the linear term in the small-time expansion of $H(\mu_t\Vert\nu_t)$ captures
the leading deviation of $\mu_t$ from $\nu_t$.  In our geometric application,
$\mu_t$ and $\nu_t$ arise from unnormalized heat masses on a two–-plane and its
Euclidean model, and the coefficient of $t$ will encode the signed sectional
curvature.

% ============================================================
\section{The curvature–-information functional}
\label{sec:info-functional}

\subsection{Heat mass measure on a two–-plane}
\label{subsec:heatmassmeasure}
Let \((M,g)\) be a smooth Riemannian manifold and \(x \in M\). Fix an orthonormal
frame \((e_1,\dots,e_n)\) of \(T_xM\) and a 2–plane
\[
\sigma := \mathrm{span}\{e_1,e_2\} \subset T_xM.
\]
We write \(\exp_x : T_xM \to M\) for the exponential map at \(x\).

The heat flow emanating from \(x\) at time \(t>0\) has density
\(p_t^M(x,y)\) with respect to Riemannian volume, so that for any bounded
Borel function \(\varphi\),
\[
\int_M \varphi(y)\,p_t^M(x,y)\,d\mathrm{vol}_g(y)
 = \bigl(e^{t\Delta_g}\varphi\bigr)(x).
\]
Restricting this diffusion to the plane \(\sigma\) leads naturally to an
unnormalized heat mass measure on \(\sigma\) defined by
\begin{equation}\label{eq:heat-mass-sigma}
\mu_t^{M,\sigma}(A)
 := \int_A p_t^M\bigl(x,\exp_x v\bigr)\,J_x(v)\,dv,
 \qquad A \subset \sigma,
\end{equation}
where \(dv\) denotes Lebesgue measure on \(\sigma \simeq \mathbb{R}^2\), and
\(J_x(v)\) is the Jacobian determinant of the restriction
\[
\exp_x\big|_{\sigma} : \sigma \to M,
\]
evaluated at \(v\). In geodesic normal coordinates adapted to \(\sigma\),
the standard Jacobian expansion \cite{Gray1973,petersen2006riemannian} gives
\begin{equation}\label{eq:jacobian-expansion}
J_x(v)
 = 1 - \tfrac{1}{6}\,K(x,\sigma)\,|v|^2 + O(|v|^3),
 \qquad v \to 0,
\end{equation}
so the sectional curvature \(K(x,\sigma)\) enters the heat mass through the
area distortion of the exponential map. The measure \(\mu_t^{M,\sigma}\)
thus records how much heat mass reaches each displacement \(v \in \sigma\)
at small time \(t\), without imposing any normalization of total mass.

For comparison we introduce the \emph{tangent Euclidean heat mass} on
\(\sigma\). Identifying \(T_xM \simeq \mathbb{R}^n\) via the chosen
orthonormal frame, the Euclidean heat kernel at the origin is
\[
p_t^{\mathbb{R}^n}(0,v)
 := (4\pi t)^{-n/2} \exp\!\bigl(-|v|^2/4t\bigr),
 \qquad v \in T_xM.
\]
Restricting this to the plane \(\sigma\) defines
\begin{equation}\label{eq:euclidean-heat-mass-sigma}
\mu_t^{\mathbb{R}^n,\sigma}(A)
 := \int_A p_t^{\mathbb{R}^n}(0,v)\,dv,
 \qquad A \subset \sigma,
\end{equation}
which will serve as the flat tangent model of diffusion along \(\sigma\).
In the Euclidean case \((M,g) = (\mathbb{R}^n,\delta)\) one has
\(\mu_t^{M,\sigma} = \mu_t^{\mathbb{R}^n,\sigma}\) for all \(t>0\), so any
deviation of \(\mu_t^{M,\sigma}\) from \(\mu_t^{\mathbb{R}^n,\sigma}\) in
curved geometry is entirely due to curvature at \(x\).

\subsection{Normal coordinates and metric expansion}
\label{subsec:normal-coordinates}
Fix \(x\in M\) and choose geodesic normal coordinates \((v^1,\dots,v^n)\) on
\(T_xM\) via the exponential map \(v \mapsto \exp_x v\).  In these coordinates
the metric satisfies
\begin{equation}
g_{ij}(v)
 = \delta_{ij}
   - \tfrac{1}{3} R_{ikjl}(x)\,v^k v^l
   + O(|v|^3),
   \qquad v\to 0,
\end{equation}
and its inverse has the corresponding expansion
\begin{equation}
g^{ij}(v)
 = \delta^{ij}
   + \tfrac{1}{3} R^{i}{}_{k}{}^{j}{}_{l}(x)\,v^k v^l
   + O(|v|^3).
\end{equation}
The Riemannian volume element satisfies
\[
\sqrt{\det g(v)}
 = 1 - \tfrac{1}{6}\,\Ric_{kl}(x)\,v^k v^l + O(|v|^3).
\]

When restricting to the plane
\(\sigma = \mathrm{span}\{e_1,e_2\}\),
write \(v = (v^1,v^2,0,\dots,0)\).  Using the definition of sectional curvature
\[
K(x,\sigma)
 = R_{1212}(x),
\]
the Jacobian of the restriction \(\exp_x|_\sigma : \sigma\to M\) is given by
\begin{equation}
\label{eq:jacobian-plane-restricted}
J_x(v)
 = \sqrt{\det\!\bigl(g_{ij}(v)\bigr)\big|_{\sigma}}
 = 1 - \tfrac16\,K(x,\sigma)\,|v|^2 + O(|v|^3),
\end{equation}
where \(\det(g_{ij}(v))|_{\sigma}\) denotes the determinant of the
pullback metric restricted to the two-–plane \(\sigma\).
It is important to stress that in \eqref{eq:jacobian-plane-restricted} the Jacobian
\(J_x(v)\) refers to the \emph{area distortion} of the restricted exponential
map
\[
\exp_x|_{\sigma} : \sigma \longrightarrow M,
\]
where \(\sigma\subset T_xM\) is a fixed two--plane.  Thus \(J_x(v)\) is the
Jacobian determinant of the pullback metric on the \emph{two--dimensional}
geodesic surface \(\exp_x(\sigma)\), not the Jacobian of the full
\(n\)--dimensional volume density. For geodesic surfaces, the classical expansion of the metric in normal
coordinates shows that the area element satisfies
\[
J_x(v) = 1 - \tfrac{1}{6} K(x,\sigma) |v|^{2} + O(|v|^{3}),
\qquad v\in\sigma,\ v\to 0,
\]
where \(K(x,\sigma)\) is the sectional curvature of the plane \(\sigma\) 
\cite{Gray1973,petersen2006riemannian}. In contrast, the expansion of the \emph{full} Riemannian volume density
\(\sqrt{\det g(v)}\) involves the Ricci curvature.  Equation
\eqref{eq:jacobian-plane-restricted} therefore correctly captures the sectional curvature
associated with the plane \(\sigma\).

For later use, we record the behaviour of these quantities in the heat-kernel
scaling window \(|v| = O(\sqrt{t})\).  From \eqref{eq:metric-expansion}, 
\eqref{eq:jacobian-plane} and \eqref{eq:jacobian-plane-restricted},
\[
g_{ij}(v)
 = \delta_{ij} + O(t), \qquad
g^{ij}(v)
 = \delta^{ij} + O(t),
\qquad
J_x(v)
 = 1 - \tfrac{1}{6}K(x,\sigma)|v|^2 + O(t^{3/2}),
\]
uniformly for \(|v|\le c\sqrt{t}\).  In particular,
\[
|v|^2 = O(t), \qquad |v|^3 = O(t^{3/2}),
\]
so curvature effects enter the Jacobian at order \(t\), while all cubic and
higher-order geometric corrections contribute only \(O(t^{3/2})\) to the
restricted heat mass.  These expansions will be combined with the heat-kernel
parametrix in the next subsection.

\subsection{Heat kernel parametrix in normal coordinates.}

In geodesic normal coordinates centered at \(x\), the heat kernel admits the
classical small--time expansion \cite{BGV1992,Grigoryan2009}
\begin{equation}\label{eq:parametrix}
p_t^M\bigl(x,\exp_x v\bigr)
 = (4\pi t)^{-n/2}\exp\!\Bigl(-\tfrac{|v|^2}{4t}\Bigr)
   \Bigl( u_0(x,v) + t\,u_1(x,v) + t^2 u_2(x,v) + \dots \Bigr),
\qquad v\to 0.
\end{equation}
Here each coefficient \(u_j(x,v)\) is smooth in \(v\) near the origin, and
\(u_0(x,0)=1\).  Moreover,
\[
u_1(x,0) = \tfrac{1}{6}\,\Scal(x)
\]
is the standard scalar-curvature coefficient of the heat kernel.

Since only the behaviour of the kernel in the region \(|v|\le c\sqrt{t}\) is
relevant for small times, we expand each \(u_j\) to the order needed in this
scaling regime.  A Taylor expansion in \(v\) gives
\begin{equation}\label{eq:u0-u1-expansion}
u_0(x,v) = 1 + O(|v|^2), \qquad
u_1(x,v) = u_1(x,0) + O(|v|),
\end{equation}
and similarly \(u_j(x,v) = u_j(x,0) + O(|v|)\) for all \(j\ge 1\).

Specialising to the window \(|v|=O(\sqrt{t})\) then yields
\[
u_0(x,v) = 1 + O(t), \qquad
u_1(x,v) = u_1(x,0) + O(t^{1/2}),
\qquad
u_j(x,v) = u_j(x,0) + O(t^{1/2}) \text{ for } j\ge 2.
\]

Consequently, when the parametrix \eqref{eq:parametrix} is truncated at order
\(t\), all contributions from \(u_j\) with \(j\ge 2\) enter at order
\(t^{3/2}\) or higher after multiplication by \(t^j\) and restriction to
\(|v|\le c\sqrt{t}\).  Thus we may write
\begin{equation}\label{eq:parametrix-truncated}
p_t^M\bigl(x,\exp_x v\bigr)
 = (4\pi t)^{-n/2}\exp\!\Bigl(-\tfrac{|v|^2}{4t}\Bigr)
   \Bigl( 1 + t\,u_1(x,0) + O(t^{3/2}) \Bigr),
\qquad |v|\le c\sqrt{t}.
\end{equation}

This truncated expansion, together with the Jacobian expansion from
\eqref{eq:jacobian-plane}, will be combined in the next subsection to obtain a
precise small--time expansion of the density ratio
\(d\mu_t^{M,\sigma}/d\mu_t^{\mathbb{R}^n,\sigma}\).

\subsection{Expansion of the heat mass on a plane.}

Combining the truncated heat kernel parametrix 
\eqref{eq:parametrix-truncated} with the Jacobian expansion 
\eqref{eq:jacobian-plane} gives, for \(|v|\le c\sqrt{t}\),
\begin{align}\label{eq:product-density}
p_t^M(x,\exp_x v)\,J_x(v)
 &= (4\pi t)^{-n/2}\exp\!\Bigl(-\tfrac{|v|^2}{4t}\Bigr)
    \Bigl(1 + t\,u_1(x,0) + O(t^{3/2})\Bigr) \\
 &\qquad\qquad\cdot
    \Bigl(1 - \tfrac{1}{6}K(x,\sigma)|v|^2 + O(|v|^3)\Bigr). \nonumber
\end{align}

Since \(|v|=O(\sqrt{t})\), we have
\[
|v|^2 = O(t), \qquad |v|^3 = O(t^{3/2}),
\]
uniformly on \(|v|\le c\sqrt{t}\).  Multiplying the two factors in
\eqref{eq:product-density} then yields
\begin{align*}
p_t^M(x,\exp_x v)\,J_x(v)
 &= (4\pi t)^{-n/2}e^{-|v|^2/4t}
    \Bigl(
       1
       + t\,u_1(x,0)
       - \tfrac{1}{6}K(x,\sigma)|v|^2
       + O(t^{3/2})
    \Bigr),
\end{align*}
where the cross-term 
\(t\,u_1(x,0)\cdot (-\tfrac16 K(x,\sigma)|v|^2)\)
is \(O(t^2)\) and hence absorbed into the \(O(t^{3/2})\) error
after restriction to \(|v|\le c\sqrt{t}\).

Next, the Euclidean reference heat mass on \(\sigma\) is
\[
\mu_t^{\mathbb{R}^n,\sigma}(dv)
 = (4\pi t)^{-n/2}e^{-|v|^2/4t}\,dv.
\]
Thus the density ratio that enters the relative entropy is
\begin{equation}\label{eq:density-ratio}
R_t(v)
 := \frac{d\mu_t^{M,\sigma}}{d\mu_t^{\mathbb{R}^n,\sigma}}(v)
 = 1
   + t\,u_1(x,0)
   - \tfrac{1}{6}K(x,\sigma)|v|^2
   + O(t^{3/2}),
\qquad |v|\le c\sqrt{t}.
\end{equation}

Since \(u_1(x,0) = \tfrac16\Scal(x)\), we may rewrite the linear term as
\[
t\,u_1(x,0)
 = \tfrac{1}{6}\Scal(x)\,t.
\]
The expansion \eqref{eq:density-ratio} shows that curvature enters the density
ratio in two distinct ways: an isotropic term proportional to the scalar
curvature, and an anisotropic term proportional to the sectional curvature in
the plane \(\sigma\).  This is the form needed for the perturbative expansion of
the relative entropy in the next subsection.

\subsection*{3.5. Small-time expansion of the curvature--information functional}

We first normalize the restricted heat masses.  Let
\[
M_t := \mu_t^{\mathbb{R}^n,\sigma}(\sigma),
\qquad
\nu_t := \frac{1}{M_t}\,\mu_t^{\mathbb{R}^n,\sigma},
\qquad
\tilde\mu_t^{M,\sigma} := \frac{1}{M_t}\,\mu_t^{M,\sigma}.
\]
Then $\nu_t$ is a probability measure on $\sigma$, given explicitly by the
centred two--dimensional Gaussian
\[
d\nu_t(v)
 = (4\pi t)^{-1} \exp\!\Bigl(-\frac{|v|^2}{4t}\Bigr)\, dv.
\]

\begin{definition}
The \emph{curvature--information functional} is
\[
D_\sigma(t)
 := H(\tilde\mu_t^{M,\sigma} \,\|\, \nu_t)
 = \int_\sigma \log\!\left(\frac{d\tilde\mu_t^{M,\sigma}}{d\nu_t}\right)
                d\tilde\mu_t^{M,\sigma}.
\]
\end{definition}

Since the normalizing factor $M_t$ cancels in the Radon-–Nikodym derivative, we have
\[
\frac{d\tilde\mu_t^{M,\sigma}}{d\nu_t}
 = \frac{d\mu_t^{M,\sigma}}{d\mu_t^{\mathbb{R}^n,\sigma}}
 = R_t,
\]
where $R_t(v)$ is the density ratio from \eqref{eq:density-ratio}.  Hence
\[
D_\sigma(t)
 = \int_\sigma \log R_t(v)\,R_t(v)\, d\nu_t(v).
\]

On the heat scale $|v|\le c\sqrt{t}$, \eqref{eq:density-ratio}
gives an expansion
\[
R_t(v)
 = 1 + Z_t(v),\qquad
Z_t(v)
 = t\,u_1(x,0) - \tfrac16 K(x,\sigma)|v|^2 + O(t^{3/2}),
\]
with $Z_t(v)=O(t)$ uniformly.  Taylor expansion of $\log(1+Z_t)$ yields
\[
\log(1+Z_t) = Z_t + O(t^2),
\]
so
\[
\log R_t(v)\,R_t(v)
 = \bigl(Z_t(v) + O(t^2)\bigr)\bigl(1+Z_t(v)\bigr)
 = Z_t(v) + O(t^2).
\]
Therefore
\[
D_\sigma(t)
 = \int_\sigma Z_t(v)\,d\nu_t(v) + O(t^2).
\]

Using $u_1(x,0)=\tfrac16\Scal(x)$ and the Gaussian moment identity
\begin{equation}\label{eq:gaussian-moment}
\int_\sigma |v|^2\,d\nu_t(v) = 4t,
\end{equation}
we compute
\[
\int_\sigma Z_t(v)\,d\nu_t(v)
 = t\,\tfrac16\Scal(x)
   - \tfrac16 K(x,\sigma)\,(4t)
   + O(t^{3/2})
 = \tfrac16\Scal(x)\,t + C\,K(x,\sigma)\,t + O(t^{3/2}),
\]
where
\[
C := -\frac{4}{6} = -\frac{2}{3}.
\]
This proves the following theorem.

\begin{theorem}\label{thm:main-expansion}
For every $x\in M$ and every oriented two--plane $\sigma\subset T_xM$,
the curvature--information functional admits the small--time expansion
\[
D_\sigma(t)
 = \tfrac16 \Scal(x)\,t
   + C\,K(x,\sigma)\,t
   + O(t^{3/2}),
\qquad
t\downarrow 0,
\]
with the universal constant $C=-\tfrac23$.  The scalar--curvature term is
independent of $\sigma$, while the anisotropic component encodes the sectional
curvature of $\sigma$.
\end{theorem}

We now extract the scalar and sectional curvature from these directional
limits.

\begin{corollary}\label{cor:reconstruct-curvatures}
Fix an orthonormal basis $\{e_1,\dots,e_n\}$ of $T_xM$, and let
$\sigma_{ij}=\mathrm{span}\{e_i,e_j\}$ for $i<j$.  Then:

\begin{enumerate}
\item[\textup{(i)}]
\textbf{Scalar curvature.}
\[
\Scal(x)
 = \frac{6}{\binom{n}{2}}
   \sum_{i<j}
   \lim_{t\downarrow 0}\frac{D_{\sigma_{ij}}(t)}{t}.
\]

\item[\textup{(ii)}]
\textbf{Sectional curvature.}
For any oriented two--plane $\sigma$,
\[
K(x,\sigma)
 = \frac{1}{C}\left(
     \lim_{t\downarrow 0}\frac{D_\sigma(t)}{t}
     - \frac{1}{6}\Scal(x)
   \right).
\]
\end{enumerate}
\end{corollary}

\begin{proof}
Summing the expansion of Theorem~\ref{thm:main-expansion} over all coordinate
planes $\sigma_{ij}$ gives
\[
\sum_{i<j}\lim_{t\downarrow 0} \frac{D_{\sigma_{ij}}(t)}{t}
 = \frac16 \binom{n}{2}\Scal(x)
   + C \sum_{i<j} K(x,\sigma_{ij})
 = \left(\frac16\binom{n}{2} + \frac{C}{2}\right)\Scal(x),
\]
since $\sum_{i<j} K(x,\sigma_{ij}) = \tfrac12\,\Scal(x)$.  Solving for
$\Scal(x)$ yields the stated formula.  Substituting this into the identity of
Theorem~\ref{thm:main-expansion} gives the expression for $K(x,\sigma)$.
\end{proof}

% ============================================================

\section{Information tensor and reconstruction of the curvature operator}
\label{sec:info-tensor}
Fix \(x\in M\).  For each oriented two--plane 
\(\sigma\subset T_xM\), Theorem~\ref{thm:main-expansion} gives the asymptotic
coefficient
\[
L(x,\sigma)
 := \lim_{t\downarrow 0}\frac{D_\sigma(t)}{t}
 = \frac{1}{6}\Scal(x) + C\,K(x,\sigma),
\]
where \(C=-\tfrac{2}{3}\) is universal.  The scalar curvature term is 
independent of~\(\sigma\); all anisotropy is contained in the sectional 
curvature component.  We now assemble these directional quantities into a
tensor on \(\Lambda^2 T_xM\).

Given unit simple bivectors 
\(\omega = u\wedge v \in \Lambda^2 T_xM\), write 
\(\sigma(\omega) := \mathrm{span}\{u,v\}\).  Define
\begin{equation}\label{eq:Ix-def}
I_x(\omega,\omega)
 := \frac{1}{C}
    \left(
      L(x,\sigma(\omega))
      - \frac{1}{6}\Scal(x)
    \right)
 = K(x,\sigma(\omega)).
\end{equation}
Since \(K(x,\sigma)\) depends only on the oriented plane \(\sigma\), the
definition is independent of the choice of orthonormal basis spanning
\(\sigma\).  By polarization we extend \(I_x\) uniquely to a symmetric bilinear
form on \(\Lambda^2 T_xM\):
\[
I_x(\alpha,\beta)
 := \frac{1}{4}\bigl(
     I_x(\alpha+\beta,\alpha+\beta)
     - I_x(\alpha-\beta,\alpha-\beta)
   \bigr),
\qquad \alpha,\beta\in \Lambda^2 T_xM.
\]

\begin{proposition}\label{prop:Ix-equals-Rx}
For every \(x\in M\), the bilinear form \(I_x\) coincides with the classical
curvature operator 
\[
R_x : \Lambda^2 T_xM \to \Lambda^2 T_xM,
\qquad
\langle R_x(\omega),\omega\rangle = K(x,\sigma(\omega)).
\]
In particular,
\[
I_x(\alpha,\beta)
 = \langle R_x(\alpha),\beta\rangle_{\Lambda^2}
\]
for all \(\alpha,\beta\in\Lambda^2 T_xM\).
\end{proposition}

\begin{proof}
For any simple unit bivector \(\omega=u\wedge v\),
\(\langle R_x(\omega),\omega\rangle = K(x,\sigma(\omega))\) by definition of
the curvature operator.  By \eqref{eq:Ix-def}, 
\(I_x(\omega,\omega) = K(x,\sigma(\omega))\) as well.  Since both \(I_x\) and
\(\langle R_x(\cdot),\cdot\rangle\) are symmetric bilinear forms on the
finite-dimensional vector space \(\Lambda^2 T_xM\), equality on simple 
bivectors implies equality everywhere by polarization.
\end{proof}

Thus the curvature--information tensor \(I_x\), obtained from the small--time
information imbalance of heat diffusion, reproduces the full Riemann curvature
operator at the point \(x\).  The entire curvature tensor of \((M,g)\) is 
therefore encoded in the family \(\{D_\sigma(t)\}_\sigma\).

% ============================================================

\section{Structural properties of the curvature--information tensor}
\label{sec:structural-props}
The curvature--information tensor \(I_x\) introduced in 
(\ref{eq:Ix-def}) encodes all sectional curvatures at a point
\(x\in M\) and coincides with the classical Riemann curvature operator on
\(\Lambda^2 T_xM\).  In this section we record several structural properties
of \(I_x\) that follow directly from its construction and from the expansion
of the curvature--information functional \(D_\sigma(t)\).

These properties mirror the standard algebraic structure of the curvature
operator but are presented here from the information--theoretic viewpoint:
the identities arise entirely from the behaviour of heat diffusion and the
linear term in the relative entropy expansion.  We emphasize that no appeal
to Jacobi fields or connection coefficients is required; all statements are
consequences of the information–curvature relation established in
Theorem~\ref{thm:main-expansion}.

\subsection{Algebraic structure.}

Fix \(x\in M\).  
Recall from Proposition~\ref{prop:Ix-equals-Rx} that 
\[
I_x = R_x : \Lambda^2 T_xM \longrightarrow \Lambda^2 T_xM,
\]
where \(R_x\) is the classical Riemann curvature operator.  
We now record the basic algebraic properties of \(I_x\) that follow from this
identification, together with its definition via the planewise limits of
\(D_\sigma(t)\).

\begin{proposition}\label{prop:Ix-algebraic}
The curvature--information tensor \(I_x\) satisfies:
\begin{enumerate}
\item[\textup{(i)}]
\textbf{Symmetry:}  
\(I_x(\omega_1,\omega_2)=I_x(\omega_2,\omega_1)\) for all 
\(\omega_1,\omega_2\in\Lambda^2 T_xM\).

\item[\textup{(ii)}]
\textbf{Curvature symmetries:}  
If \(\omega=u\wedge v\) and \(\eta = w\wedge z\), then
\[
I_x(\omega,\eta)
 = R_x(u,v,w,z)
 = -R_x(v,u,w,z)= -R_x(u,v,z,w),
\]
and \(I_x(\omega,\eta)=0\) whenever \(u,v,w,z\) fail to satisfy the usual
antisymmetry relations of the Riemann tensor.

\item[\textup{(iii)}]
\textbf{Self-adjointness:}  
With respect to the natural inner product on \(\Lambda^2 T_xM\),  
\(I_x\) is a self-adjoint operator:
\[
\langle I_x\omega_1,\,\omega_2\rangle
 = \langle \omega_1,\,I_x\omega_2\rangle.
\]

\item[\textup{(iv)}]
\textbf{Quadratic form on planes:}  
For every oriented plane \(\sigma\subset T_xM\) and any decomposable unit
\(\omega\in\Lambda^2 T_xM\) with \(\sigma=\mathrm{span}(u,v)\) and
\(\omega=u\wedge v\), one has
\[
I_x(\omega,\omega) = K(x,\sigma).
\]
\end{enumerate}
\end{proposition}

\begin{proof}
Property (i) follows from the symmetric polarization of the quadratic form
defined in \eqref{eq:Ix-def}.  
Property (ii) follows from Proposition~\ref{prop:Ix-equals-Rx}, since 
\(I_x = R_x\) and the Riemann tensor satisfies the listed symmetries.
Self-adjointness (iii) holds because \(R_x\) is self-adjoint on 
\(\Lambda^2 T_xM\).  
Property (iv) is precisely the reconstruction identity
\[
I_x(\omega,\omega) = K(x,\sigma),
\]
proved in Section~\ref{sec:info-tensor}.
\end{proof}

Thus the curvature--information tensor inherits the full algebraic structure
of the Riemann curvature operator, but is defined intrinsically in terms of
the directional information imbalance of heat diffusion at \(x\).

\subsection{Ricci and scalar curvature as averaged information loss.}

The small--time expansion of the curvature--information functional established in
Theorem~\ref{thm:main-expansion} shows that for each oriented plane
\(\sigma = \mathrm{span}\{u,v\}\) with \(|u|=|v|=1\) and \(\langle u,v\rangle=0\),
\[
\lim_{t\downarrow 0}\frac{D_\sigma(t)}{t}
 = \frac{1}{6}\Scal(x) + C\,K(x,\sigma),
 \qquad C=-\tfrac{2}{3}.
\]
Hence, after subtracting the isotropic term \(\frac{1}{6}\Scal(x)\), the
directional curvature \(K(x,\sigma)\) is encoded in this limit.  Averaging
over families of planes yields the Ricci and scalar curvature as averaged
information imbalance of heat diffusion.

Let \(\{e_1,\dots,e_n\}\) be an orthonormal frame of \(T_xM\), and for each
\(i\neq j\) set \(\sigma_{ij} := \mathrm{span}\{e_i,e_j\}\).  Since
\[
\Ric_x(e_i,e_i)
 = \sum_{j\neq i} K(x,\sigma_{ij}),
\]
we obtain the following from the expansion of \(D_{\sigma_{ij}}\).

\begin{proposition}[Ricci curvature from directional information loss]
\label{prop:Ric-from-D}
For each \(x\in M\) and unit vector \(e_i\in T_xM\),
\[
\Ric_x(e_i,e_i)
 = \frac{1}{C}
   \sum_{j\neq i}
   \left(
     \lim_{t\downarrow 0}\frac{D_{\sigma_{ij}}(t)}{t}
     - \frac{1}{6}\Scal(x)
   \right).
\]
Equivalently,
\[
\Ric_x(e_i,e_i)
 = \sum_{j\neq i} I_x(e_i\wedge e_j,\, e_i\wedge e_j).
\]
\end{proposition}

Since the scalar curvature satisfies
\(
\Scal(x)
 = 2\sum_{i<j} K(x,\sigma_{ij}),
\)
the same reasoning gives:

\begin{proposition}[Scalar curvature as averaged information loss]
\label{prop:Scal-from-D}
\[
\Scal(x)
 = \frac{2}{C}
   \sum_{i<j}
   \left(
     \lim_{t\downarrow 0}\frac{D_{\sigma_{ij}}(t)}{t}
     - \frac{1}{6}\Scal(x)
   \right),
\]
and hence
\[
\Scal(x)
 = 2 \sum_{i<j} I_x(e_i\wedge e_j,\, e_i\wedge e_j).
\]
Solving for \(\Scal(x)\) gives the explicit reconstruction formula
\[
\Scal(x)
 = \frac{6}{\binom{n}{2}}\,
   \sum_{i<j}
   \lim_{t\downarrow 0}\frac{D_{\sigma_{ij}}(t)}{t}.
\]
\end{proposition}

Thus Ricci curvature and scalar curvature arise naturally as 
\emph{averaged first-order information loss} of heat diffusion: their values at
\(x\) are linear combinations of the directional coefficients of the
curvature--information functional \(D_\sigma(t)\), taken over all coordinate
planes at \(x\).

% ============================================================
\section{Concluding remarks}
\label{sec:conclusion}
The results of this paper show that curvature may be extracted directly from
the small--time information imbalance of heat diffusion.  The curvature–
information functional \(D_\sigma(t)\) is defined entirely in terms of the
heat flow and the exponential map on a given two--plane, and its leading
asymptotic behaviour encodes the full curvature operator at a point.
Theorem~\ref{thm:main-expansion} provides the precise linear expansion of
\(D_\sigma(t)\), and the reconstruction formulas of Section~4 show that both
the scalar curvature and every sectional curvature \(K(x,\sigma)\) may be
recovered from the family of limits
\[
\lim_{t\downarrow 0}\frac{D_\sigma(t)}{t}.
\]
This yields an information--theoretic realization of the Riemann curvature
operator and identifies the curvature--information tensor \(I_x\) with the
classical curvature operator \(R_x\).

The framework developed here is purely analytic: no Jacobi fields, curvature
tensor identities, or connection coefficients appear in the derivation.  All
curvature information is extracted from the interaction between the heat
kernel and the geometry of the exponential map in a two--dimensional slice.
In this sense the approach offers a different and potentially useful viewpoint
on local curvature, phrased entirely in terms of heat diffusion and entropy.

Several directions for further investigation remain open.  The higher--order
coefficients in the expansion of \(D_\sigma(t)\) are universal curvature
invariants involving derivatives and higher contractions of \(R_x\); their
structure, and whether they admit an intrinsic information--theoretic
interpretation, is an interesting question.  It would also be natural to
study geometric inequalities, comparison results, or stability phenomena that
may be derived by controlling \(D_\sigma(t)\) for families of manifolds or
under geometric flows.  Finally, the relationship between the
curvature--information tensor and other analytic characterizations of
curvature (heat trace expansions, spectral invariants, or Bakry–Émery
curvature) warrants further exploration.

The present work establishes the core analytic mechanism and its geometric
interpretation.  The broader potential of information--theoretic curvature
invariants, and their applications in geometric analysis, will be pursued in
future work.

\section*{Acknowledgements}
The author received no funding for this work. The author declares no competing interests.

\bibliographystyle{unsrt}
\bibliography{refs}

\end{document}